\documentclass[10pt,twocolumn,english]{IEEEtran}
\usepackage[T1]{fontenc}
\usepackage[latin9]{inputenc}
\usepackage{units}
\usepackage{amsthm,amsmath,amssymb}
\usepackage{graphicx,xcolor}
\usepackage{setspace}
\usepackage{caption}
\usepackage[labelformat=simple]{subcaption}
\usepackage{algorithm, algorithmic}
\usepackage[hidelinks]{hyperref}
\usepackage{babel}
\usepackage[numbers]{natbib}

\graphicspath{{./Figs/}}

\newtheorem{defn}{Definition}
\newtheorem{assumption}{Assumption}
\newtheorem{thm}{Theorem}

\newtheorem{cor}{Corollary}
\newtheorem{lem}{Lemma}
\newtheorem{rem}{Remark}

\newcommand{\K}{\mathcal{K}}
\newcommand{\Kinf}{\K_{\infty}}
\newcommand{\LL}{\mathcal{L}}
\newcommand{\KL}{\mathcal{KL}}
\newcommand{\KdL}{\mathcal{K \cdot L}}

\begin{document}

\title{\huge Optimization-based State Estimation under Bounded Disturbances}

\author{Wuhua Hu, Lihua Xie, and Keyou You
\thanks{This research was supported in part by the National High Technology Research and Development Program (863 Program) of China (2013AA040703). The first author thanks Professor James B. Rawlings for useful discussion on an initial version of this work, and Ji Luo for his useful comments.}
\thanks{W. Hu and L. Xie are with the School of Electrical and Electronic Engineering, Nanyang Technological University, Singapore. {\tt\small  \{hwh, elhxie\}@ntu.edu.sg}}
\thanks{ K. You is with the Department of Automation and TNList,
Tsinghua University,  Beijing 100084, P.~R.~China. {\tt\small  youky@tsinghua.edu.cn}}
}

\maketitle
\begin{abstract}
This paper studies an optimization-based state estimation approach for discrete-time nonlinear systems under bounded process and measurement disturbances. We first introduce a full information estimator (FIE), which is given as a solution to minimize a cost function by using all the available measurements. Then, we prove that the FIE of an incrementally input/output-to-state stable system is robustly globally asymptotically stable under a certain class of cost functions. Moreover, the implications and relationships with related results in the literature are discussed. Finally, a simple example is included to illustrate the theoretical results. 
\end{abstract}


\section{Introduction}

Optimization-based estimation and in particular, a moving-horizon estimator
(MHE) has attracted extensive attention recently \cite{wynn2014convergence,ellis2014robust,voelker2013moving,rawlings2012optimization,alessandri2008moving,rawlings2006particle,rao2003constrained}.
MHE only uses the latest information to do optimization-based estimation, and has advantages in nonlinear systems over classical approaches such as extended Kalman
filter (EKF) \cite{ljung1979asymptotic,haseltine2005critical}. In contrast, a full information version (FIE) of the optimization-based estimator uses all the historical information for the state estimation. Although FIE is generically
intractable, it is fundamentally important as it provides a performance benchmark for other estimators \cite{rawlings2012optimization}.

Fundamental results on FIE were recently reviewed in
\cite{rawlings2012optimization}. When the system is incrementally input/output-to-state stable (i-IOSS), an FIE is robustly globally asymptotically stable (RGAS) for {\em convergent} process
and measurement disturbances if the cost function of the FIE satisfies certain conditions. However, it is unclear under what
conditions the above conclusion still holds for \emph{bounded} process
and measurement disturbances, which obviously happens more often in practice. The authors posted this challenge as
an open problem in the review paper.

This paper provides sufficient conditions for an FIE to be RGAS under bounded disturbances. The conditions require an appropriately defined cost function for the FIE, while the system is assumed to be i-IOSS. The general conditions become specific ones for three special cases, including the one investigated in \cite{ji2013robust} (a work inspired by Theorem \ref{thm:(RGAS-of-FIE} of this paper). We also note that the FIE having the RGAS property may be viewed as a kind of state observer for nonlinear systems which has attracted continuous attention and been researched for a long time \cite{sontag1997output,sontag2008input}.

The rest of the paper is organized as follows. In Section \ref{sec:Full-information-estimation},
we introduce notation and define the FIE under bounded disturbances. In Section \ref{sec:New-results-with},
we present sufficient conditions for the FIE to be RGAS, followed by a numerical example in Section \ref{sec:numerical example}. Finally, we draw conclusions in Section \ref{sec:Conclusion}.

\section{Full Information Estimation\label{sec:Full-information-estimation}}

We adopt the notation used in \cite{rawlings2012optimization} for the problem formulation. The symbols $\mathbb{R}$, $\mathbb{R}_{\ge0}$
and $\mathbb{I}_{\ge0}$ denote the sets of real numbers, nonnegative real numbers
and nonnegative integers, respectively; and $\mathbb{I}_{a : b}$ denotes the set of integers from $a$ to $b$. The symbol $\left|\cdot\right|$
denotes the Euclidean norm. The bold symbol $\boldsymbol{x}$, denotes
a sequence of vector-valued variables $x$, $\{x(0),\, x(1),\,...\}$.
The notation $\left\Vert \boldsymbol{x}\right\Vert $ is the supreme norm
over a sequence, $\sup_{i\ge0}\left|x(i)\right|$, and $\left\Vert \boldsymbol{x}\right\Vert _{a:b}$
denotes $\max_{a\le i\le b}\left|x(i)\right|$. The frequently used $\K$, $\Kinf$, $\LL$ and $\KL$ functions are defined as follows.
\begin{defn}
($\K$, $\Kinf$, $\LL$ and $\KL$ functions)
A function $\alpha: \mathbb{R}_{\ge0} \to \mathbb{R}_{\ge0}$
is a $\K$-function if it is continuous, zero at zero,
and strictly increasing, and a $\Kinf$-function if $\alpha$ is a $\K$-function and satisfies $\alpha(s) \to \infty$ as $s \to \infty$. A function $\varphi: \mathbb{R}_{\ge 0} \to \mathbb{R}_{\ge 0}$ is a $\LL$-function if it is continuous, nonincreasing and satisfies $\varphi(t) \to 0$ as $t \to \infty$. A function $\beta: \mathbb{R}_{\ge 0} \times \mathbb{R}_{\ge 0} \to \mathbb{R}_{\ge0}$ is a $\KL$-function if, for each
$t \ge 0$, $\beta(\cdot, t)$ is a $\K$-function and for each $s \ge 0$, $\beta(s, \cdot)$ is a $\LL$-function.
\end{defn}
The following properties of the $\K$- and $\KL$-functions will be used in proving our main results.
\begin{lem} \cite{rawlings2012optimization}
Given a $\K$-function $\alpha$ and a $\KL$-function
$\beta$, the following holds for all $a_{i} \in \mathbb{R}_{\ge 0}$,
$i \in \mathbb{I}_{1 : n}$, and all $t \in \mathbb{R}_{\ge 0}$,
\begin{equation*}
\alpha\left(\sum_{i=1}^{n}a_{i}\right) \le  \sum_{i=1}^{n}\alpha(na_{i}),\,\,
\beta\left(\sum_{i=1}^{n}a_{i}, t\right) \le \sum_{i=1}^{n}\beta(na_{i}, t).
\end{equation*}
\end{lem}
In this work, we consider a discrete-time
nonlinear system described by
\begin{equation}
x^{+} = f(x, w),\,\,y = h(x) + v,
\label{eq:system}
\end{equation}
where $x \in \mathbb{R}^{n}$ is the system state, $y \in \mathbb{R}^{p}$
the measurement, $w \in \mathbb{R}^{g}$ the process disturbance, $v \in \mathbb{R}^{p}$
the measurement disturbance, and $x^{+} \in \mathbb{R}^{n}$ the system
state at the next sample time. A control input known up to the present time can be included but can be ignored in the formulation for the state
estimation \cite{rawlings2012optimization}. The functions $f$ and
$h$ are assumed to be continuous and known, and the initial state $x(0)$ and the disturbances $(w,v)$ are modeled as unknown but \emph{bounded} variables, which covers convergent disturbances as a special case.  

The state estimation problem is to find an optimal estimator of state
$\boldsymbol{x}$ based on measurement $\boldsymbol{y}$ as recorded for all sampled
times. This can be formulated as an optimization problem, yielding the so-called FIE. Let the decision variables be $(\boldsymbol{\chi}, \boldsymbol{\omega}, \boldsymbol{\nu})$,
which correspond to the system variables $(\boldsymbol{x}, \boldsymbol{w}, \boldsymbol{v})$,
and the optimal decision variables be $(\hat{\boldsymbol{x}}, \hat{\boldsymbol{w}}, \hat{\boldsymbol{v}})$.
Since $(\hat{\boldsymbol{x}}, \hat{\boldsymbol{w}}, \hat{\boldsymbol{v}})$,
which consist of optimal estimates at all sampled times, are uniquely
determined once $\hat{x}(0)$ and $\hat{\boldsymbol{w}}$ are known,
the decision variables essentially reduce to $\chi(0)$ and $\boldsymbol{\omega}$.
Let $t$ be the current time and $\bar{x}_{0}$ be the prior information
for the initial state. The uncertainty in the initial state is thus
denoted by $\chi(0) - \bar{x}_{0}$. Denote the cost function as $V_{t}(\chi(0) - \bar{x}_{0}, \boldsymbol{\omega})$, which
penalizes uncertainties in both the initial state and the
process. Then the FIE is defined as an optimization problem:
\begin{equation}
\begin{aligned}
\mathrm{FIE:~~~~~~} & \inf_{\chi_0, \boldsymbol{\omega}}V_{t}(\chi(0) - \bar{x}_{0}, \boldsymbol{\omega})\\
\mathrm{subject~to,~} & \chi^{+} = f(\chi, \omega), \,\,y = h(\chi) + \nu,\\
& \boldsymbol{\omega} \in \mathbb{B}_{w}, \boldsymbol{\nu} \in \mathbb{B}_{v},
\end{aligned}
\label{eq: FIE}
\end{equation}
where $\boldsymbol{\omega}$ and $\boldsymbol{\nu}$ denote the sequences of variables $\{\omega(i)\}$ and $\{\nu(i)\}$ for ${i \in \mathbb{I}_{0 : t}}$, respectively, and $\mathbb{B}_{w}$ and $\mathbb{B}_{v}$ denote two sets of bounded sequences of disturbances. If the actual disturbances are further known to converge to zero, then the two sets denote sequences of the convergent disturbances.

One important problem with the FIE is to identify conditions under
which the above optimization has an optimal solution for $(\chi(0),\,\boldsymbol{\omega})$
such that the state estimate satisfies the RGAS property defined
below. Let $\boldsymbol{x}(x_{0}, \boldsymbol{w})$ denote a state
sequence with an initial condition $x(0)=x_{0}$, and
a disturbance sequence $\boldsymbol{w}=\{w(0), w(1), ...\}$.
\begin{defn}
(RGAS \cite{rawlings2012optimization}) The estimate is based on the noisy measurement $\boldsymbol{y} = h(\boldsymbol{x}(x_{0}, \boldsymbol{w})) + \boldsymbol{v}$.
The estimate is RGAS if for all $x_{0}$ and $\bar{x}_{0}$, and bounded
$(\boldsymbol{w}, \boldsymbol{v})$, there exist functions $\beta_{x} \in \KL$
and $\alpha_{w}, \alpha_{v} \in \K$ such that the following inequality
holds for all $t \in \mathbb{I}_{\ge 0}$:
\begin{align*}
& \left|x(t; x_{0}, \boldsymbol{w}) - x(t; \hat{x}(0|t), \hat{\boldsymbol{w}}_{t})\right|\\
& \le \beta_{x}(\left|x_{0} - \bar{x}_{0}\right|, t) + \alpha_{w}(\left\Vert \boldsymbol{w}\right\Vert _{0 : t-1}) + \alpha_{v}(\left\Vert \boldsymbol{v}\right\Vert _{0 : t}),
\end{align*}
in which $\hat{x}(0|t)$ and $\hat{\boldsymbol{w}}_{t}$ are respectively
the initial state estimate and the estimated disturbances using measurements up to time $t$, and $x(t; x_{0}, \boldsymbol{w})$ denotes the system state of (\ref{eq:system}) at time $t$ with the initial state $x_0$ and disturbance $\boldsymbol{w}$.
\end{defn}
Note that we also consider the current measurement $y(t)$, which is ignored in the original definition \cite{rawlings2012optimization}. To obtain an RGAS FIE, the cost function needs to appropriately penalize the uncertainties in the initial state and the system, and the system dynamics should satisfy certain conditions. We identify and present such sufficient conditions in the next section.

\section{RGAS of the FIE\label{sec:New-results-with}}

We first introduce two definitions and one useful lemma as used in the sequel.
\begin{defn}
(i-IOSS \cite{rawlings2012optimization,sontag1997output})
The system $x^{+} = f(x, w)$, $y = h(x)$ is i-IOSS if there exist
functions $\beta \in \KL$ and $\alpha_{1}, \alpha_{2} \in \K$
such that for every two initial states $x_{01}, x_{02}$, and
two disturbances $\boldsymbol{w}_{1}, \boldsymbol{w}_{2}$, the following
holds:
\begin{align}
&\left|x(t; x_{01}, \boldsymbol{w}_{1})-x(t; x_{02}, \boldsymbol{w}_{2})\right| \le \beta(\left|x_{01} - x_{02}\right|, t) \nonumber\\
&+\alpha_{1}(\left\Vert \boldsymbol{w}_{1} - \boldsymbol{w}_{2}\right\Vert _{0 : t-1}) + \alpha_{2}(\left\Vert \boldsymbol{y}_{1} - \boldsymbol{y}_{2}\right\Vert _{0 : t}), \forall t\in\mathbb{I}_{\ge 0}.\label{eq:definition - i-IOSS}
\end{align}
\end{defn}
The definition of i-IOSS can be interpreted as a ``detectability'' concept
for nonlinear systems \cite{sontag1997output}, as the state may
be ``detected'' from the \emph{noise-free} output by (\ref{eq:definition - i-IOSS}).

In particular, if in \eqref{eq:definition - i-IOSS} $\beta(s, t) = \alpha(s) a^t$ for all $s, t \ge 0$,
with $\alpha \in \K$ and $a$ being a constant within $(0,1)$, the system is said to be \emph{exponentially i-IOSS} or \emph{exp-i-IOSS} for short. This can be viewed as extending the exponential input-to-state stability \cite{grune1999input,liu2010exponential} to the context of i-IOSS.

\begin{defn}\label{def: (-factorizable-}
($\KdL$-function) A $\KL$-function $\beta$ is called a $\KdL$-function if there exist functions $\alpha \in \K$ and $\varphi \in \LL$ such that $\beta(s, t)=\alpha(s)\varphi(t)$,
for all $s, t \ge 0$.
\end{defn}

As an example, the $\KL$-function $se^{-t}$ is a $\KdL$-function for $s,t\ge0$. The next lemma shows the general interest of a $\KdL$-function.
\begin{lem} \label{lem: factorizable-KL-bound}
($\KdL$ bound) Given an arbitrary $\KL$-function $\beta$, there exists a $\KdL$-function $\bar{\beta}$ such that $\beta(s, t) \le \bar{\beta}(s, t)$ for all $s, t \ge 0$.
\end{lem}
\begin{proof}
By Lemma 8 in \cite{sontag1998comments}, given arbitrary $\beta \in \KL$, there exist two functions $\alpha_1, \alpha_2 \in \Kinf$ such that $\beta(s, t) \le \alpha_1(s)\alpha_2(e^{-t}) =: \bar{\beta}(s, t)$ for all $s, t \ge 0$. Since $\bar{\beta}(s, t)$ is a $\KdL$-function in $s$ and $t$, this completes the proof.
\end{proof}

Lemma \ref{lem: factorizable-KL-bound} implies that the i-IOSS property in \eqref{eq:definition - i-IOSS} can be defined equivalently using a $\KdL$-function, which is useful in our later stability analysis of the FIE. Next, we introduce two assumptions for establishing our main result.

\begin{assumption} \label{assump: A1}
The FIE's cost function, $V_{t}(\chi(0) - \bar{x}_{0}, \boldsymbol{\omega})$, is defined to be continuous and satisfy the following inequalities for all $\chi(0), \bar{x}_{0} \in \mathbb{R}^{n}$,
$\boldsymbol{\omega} \in \mathbb{B}_{w}$ and $\boldsymbol{\nu} \in \mathbb{B}_{v}$:
\begin{align}
& \underbar{\ensuremath{\rho}}_{x}(\left|\chi(0) - \bar{x}_{0}\right|,  t) + \underbar{\ensuremath{\gamma}}_{w}(\left\Vert \boldsymbol{\omega}\right\Vert _{0 : t-1}) + \underbar{\ensuremath{\gamma}}_{v}(\left\Vert \boldsymbol{\nu}\right\Vert _{0 : t}) \nonumber \\
& \le V_{t}(\chi(0) - \bar{x}_{0}, \boldsymbol{\omega}) \nonumber\\
& \le \rho_{x}(\left|\chi(0) - \bar{x}_{0}\right|, t) + \gamma_{w}(\left\Vert \boldsymbol{\omega}\right\Vert _{0 : t-1}) + \gamma_{v}(\left\Vert \boldsymbol{\nu}\right\Vert _{0 : t}),\label{eq: Assumption 1}
\end{align}
where $\underbar{\ensuremath{\rho}}_{x}, \rho_{x} \in \KL$
and $\underbar{\ensuremath{\gamma}}_{w}, \underbar{\ensuremath{\gamma}}_{v}, \gamma_{w}, \gamma_{v} \in \Kinf$.
\end{assumption}

\begin{assumption} \label{assump: A2}
The $\K$ and $\KL$
functions in (\ref{eq:definition - i-IOSS})-(\ref{eq: Assumption 1})
satisfy the following inequalities for all $s_{x}, s_{w}, s_{v}, t \ge 0$:
\begin{align*}
& \beta\left(s_{x} + \underbar{\ensuremath{\gamma}}_{x, t}^{-1}\left(\rho_{x}(s_{x},  t) + \gamma_{w}(s_{w}) + \gamma_{v}(s_{v})\right), t\right) \\
& \le \bar{\beta}_{x}(s_{x}, t) + \bar{\alpha}_{w}(s_{w}) + \bar{\alpha}_{v}(s_{v}),
\end{align*}
in which $\underbar{\ensuremath{\gamma}}_{x, t}(s) := \underbar{\ensuremath{\rho}}_{x}(s, t)$
and $\underbar{\ensuremath{\gamma}}_{x, t}^{-1}(\cdot)$ defines
its inverse function, and $\bar{\beta}_{x}$, $\bar{\alpha}_{w}$
and $\bar{\alpha}_{v}$ are proper $\KL$, $\K$
and $\K$ functions, respectively.
\end{assumption}

Assumption \ref{assump: A1} ensures the FIE to have a property resembling the i-IOSS property of
the system, and Assumption \ref{assump: A2} ensures the FIE to be more sensitive to the initial state than the system to be. These will be clearer in the following corollaries. Under the above two assumptions, we establish our main result.

\begin{thm} \label{thm:(RGAS-of-FIE}(RGAS of the FIE) Suppose that the infimum in (\ref{eq: FIE}) is attainable, and the system (\ref{eq:system}) is i-IOSS. Under Assumptions \ref{assump: A1}-\ref{assump: A2}, the FIE in (\ref{eq: FIE}) is RGAS. Moreover, if we know that the disturbances ${w}(t)$ and ${v}(t)$ converge to zero in time,  then the FIE converges to the true state as $t \to \infty$.
\end{thm}
\begin{proof} (a)
\emph{ RGAS.}  Let the global optimal solution of the FIE result in a minimum cost $V_{t}^{o}$. It follows that for all $t\ge0$,
\[
V_{t}^{o}=V_{t}(\hat{x}(0|t)-\bar{x}_{0},\,\hat{\boldsymbol{w}}_{t})\le V_{t}(x_{0}-\bar{x}_{0},\,\boldsymbol{w})=:\bar{V}_{t}.
\]
By Assumption \ref{assump: A1} we have
\[
\bar{V}_{t} \le \rho_{x}(\left|x_{0}-\bar{x}_{0}\right|,\, t)+\gamma_{w}(\left\Vert \boldsymbol{w}\right\Vert _{0:t-1})+\gamma_{v}(\left\Vert \boldsymbol{v}\right\Vert _{0:t}).
\]
Together with $\underbar{\ensuremath{\gamma}}_{x,\, t}(\left|\hat{x}(0|t)-\bar{x}_{0}\right|):=\underbar{\ensuremath{\rho}}_{x}(\left|\hat{x}(0|t)-\bar{x}_{0}\right|,\, t)\le V_{t}^{o}\le\bar{V}_{t}$,
this leads to $\left|\hat{x}(0|t)-\bar{x}_{0}\right|\le\underbar{\ensuremath{\gamma}}_{x,\, t}^{-1}(\bar{V}_{t})$,
where $\underbar{\ensuremath{\gamma}}_{x,\, t}^{-1}(\bar{V}_{t})$
is dependent on time $t$. By using the triangle inequality, this
further results in
\begin{gather}
\left|x_{0}-\hat{x}(0|t)\right|\le\left|x_{0}-\bar{x}_{0}\right|+\left|\hat{x}(0|t)-\bar{x}_{0}\right|\nonumber \\
\le\left|x_{0}-\bar{x}_{0}\right|+\underbar{\ensuremath{\gamma}}_{x,\, t}^{-1}\left(\begin{array}{c}
\rho_{x}(\left|x_{0}-\bar{x}_{0}\right|,\, t)\\
+\gamma_{w}(\left\Vert \boldsymbol{w}\right\Vert _{0:t-1})+\gamma_{v}(\left\Vert \boldsymbol{v}\right\Vert _{0:t})
\end{array}\right).\label{eq:x0-x0^}
\end{gather}

The second term on right hand side of the second inequality
is dependent on time $t$.

Next we derive a bound for the term $\left\Vert \boldsymbol{w}-\hat{\boldsymbol{w}}_{t}\right\Vert _{0:t-1}$.
From the triangle inequality we have
\begin{equation}\label{equ1}
\left\Vert \boldsymbol{w}-\hat{\boldsymbol{w}}_{t}\right\Vert _{0:t-1}\le\left\Vert \boldsymbol{w}\right\Vert _{0:t-1}+\left\Vert \hat{\boldsymbol{w}}_{t}\right\Vert _{0:t-1}.
\end{equation}
Since $\underbar{\ensuremath{\gamma}}_{w}(\left\Vert \hat{\boldsymbol{w}}_{t}\right\Vert _{0:t-1})\le V_{t}^{o}\le\bar{V}_{t}$,
it implies that $\left\Vert \hat{\boldsymbol{w}}_{t}\right\Vert _{0:t-1}\le\underbar{\ensuremath{\gamma}}_{w}^{-1}(\bar{V}_{t})$.
Consequently, it follows from (\ref{equ1}) that
\begin{align}
 &  \left\Vert \boldsymbol{w}-\hat{\boldsymbol{w}}_{t}\right\Vert _{0:t-1}\nonumber \\
 & \le \left\Vert \boldsymbol{w}\right\Vert _{0:t-1}+\underbar{\ensuremath{\gamma}}_{w}^{-1}\left(\begin{array}{c}
\rho_{x}(\left|x_{0}-\bar{x}_{0}\right|,\, t)\\
+\gamma_{w}(\left\Vert \boldsymbol{w}\right\Vert _{0:t-1})+\gamma_{v}(\left\Vert \boldsymbol{v}\right\Vert _{0:t})
\end{array}\right)\nonumber \\
 & \le \rho{}_{x}^{w}(\left|x_{0}-\bar{x}_{0}\right|,\, t)+\left\Vert \boldsymbol{w}\right\Vert _{0:t-1}\nonumber \\
 & \,\,\,\,\,\,+\gamma_{w}^{w}(\left\Vert \boldsymbol{w}\right\Vert _{0:t-1})+\gamma_{v}^{w}(\left\Vert \boldsymbol{v}\right\Vert _{0:t}),\label{eq:w-w^}
\end{align}
where $\rho{}_{x}^{w}:=\underbar{\ensuremath{\gamma}}_{w}^{-1}\circ\rho_{x}$,
$\gamma{}_{w}^{w}:=\underbar{\ensuremath{\gamma}}_{w}^{-1}\circ\gamma{}_{w}$
and $\gamma{}_{v}^{w}:=\underbar{\ensuremath{\gamma}}_{w}^{-1}\circ\gamma{}_{v}$
which are $\KL$, $\K$ and $\K$ functions,
respectively. By applying the same reasoning to $\left\Vert \boldsymbol{v}-\hat{\boldsymbol{v}}_{t}\right\Vert _{0:t-1}$,
it yields
\begin{align}
\left\Vert \boldsymbol{v}-\hat{\boldsymbol{v}}_{t}\right\Vert _{0:t}& \le \rho{}_{x}^{v}(\left|x_{0}-\bar{x}_{0}\right|,\, t) \nonumber \\
& \,\,\,\,\,\,+\left\Vert \boldsymbol{v}\right\Vert _{0:t}+\gamma_{w}^{v}(\left\Vert \boldsymbol{w}\right\Vert _{0:t-1})+\gamma_{v}^{v}(\left\Vert \boldsymbol{v}\right\Vert _{0:t}),\label{eq:v-v^}
\end{align}
where $\rho{}_{x}^{v}$, $\gamma{}_{w}^{v}$ and $\gamma{}_{v}^{v}$
are $\KL$, $\K$ and $\K$ functions,
respectively.

Substitute (\ref{eq:x0-x0^})-(\ref{eq:v-v^}) into (\ref{eq:definition - i-IOSS}) of i-IOSS leads to (\ref{eq: key eq}),
\begin{figure*}
\begin{align}
 &  \left|x(t;\, x_{0},\,\boldsymbol{w})-x(t;\,\hat{x}(0|t),\,\hat{\boldsymbol{w}}_{t})\right|\nonumber \\
 & \le \beta\left(\left|x_{0}-\bar{x}_{0}\right|+\underbar{\ensuremath{\gamma}}_{x,\, t}^{-1}\left(\rho_{x}(\left|x_{0}-\bar{x}_{0}\right|,\, t)+\gamma_{w}(\left\Vert \boldsymbol{w}\right\Vert _{0:t-1})+\gamma_{v}(\left\Vert \boldsymbol{v}\right\Vert _{0:t})\right),\, t\right)\nonumber \\
 &  \,\,\,\,\,\,+\alpha_{1}\left(\rho{}_{x}^{w}(\left|x_{0}-\bar{x}_{0}\right|,\, t)+\left\Vert \boldsymbol{w}\right\Vert _{0:t-1}+\gamma_{w}^{w}(\left\Vert \boldsymbol{w}\right\Vert _{0:t-1})+\gamma_{v}^{w}(\left\Vert \boldsymbol{v}\right\Vert _{0:t})\right)\nonumber \\
 &  \,\,\,\,\,\,+\alpha_{2}\left(\rho{}_{x}^{v}(\left|x_{0}-\bar{x}_{0}\right|,\, t)+\left\Vert \boldsymbol{v}\right\Vert _{0:t}+\gamma_{w}^{v}(\left\Vert \boldsymbol{w}\right\Vert _{0:t-1})+\gamma_{v}^{v}(\left\Vert \boldsymbol{v}\right\Vert _{0:t})\right)\nonumber \\
 & \le \bar{\beta}_{x}\left(\left|x_{0}-\bar{x}_{0}\right|,\, t\right)+\alpha_{1}\left(4\rho{}_{x}^{w}(\left|x_{0}-\bar{x}_{0}\right|,\, t)\right)+\alpha_{2}\left(4\rho{}_{x}^{v}(\left|x_{0}-\bar{x}_{0}\right|,\, t)\right)\nonumber \\
 &  \,\,\,\,\,\,+\bar{\alpha}_{w}\left(\left\Vert \boldsymbol{w}\right\Vert _{0:t-1}\right)+\alpha_{1}\left(4\left\Vert \boldsymbol{w}\right\Vert _{0:t-1}\right)+\alpha_{1}\left(4\gamma_{w}^{w}(\left\Vert \boldsymbol{w}\right\Vert _{0:t-1})\right)+\alpha_{2}\left(4\gamma_{w}^{v}(\left\Vert \boldsymbol{w}\right\Vert _{0:t-1})\right)\nonumber \\
 &  \,\,\,\,\,\,+\bar{\alpha}_{v}\left(\left\Vert \boldsymbol{v}\right\Vert _{0:t}\right)+\alpha_{1}\left(4\gamma_{v}^{w}(\left\Vert \boldsymbol{v}\right\Vert _{0:t})\right)+\alpha_{2}\left(4\left\Vert \boldsymbol{v}\right\Vert _{0:t}\right)+\alpha_{2}\left(4\gamma_{v}^{v}(\left\Vert \boldsymbol{v}\right\Vert _{0:t})\right),\label{eq: key eq}
\end{align}
\end{figure*}
where Assumption \ref{assump: A2} has been used to derive the last inequality. As the sums of terms in the three lines of the
last inequality form classes $\KL$, $\K$ and $\K$
functions, respectively, we can denote them as $\beta_{x}(\left|x_{0}-\bar{x}_{0}\right|,\, t)$,
$\alpha_{w}(\left\Vert \boldsymbol{w}\right\Vert _{0:t-1})$ and $\alpha_{v}(\left\Vert \boldsymbol{v}\right\Vert _{0:t})$ in sequence, and hence conclude from (\ref{eq: key eq}) that
\begin{align*}
& \left|x(t;\, x_{0},\,\boldsymbol{w})-x(t;\,\hat{x}(0|t),\,\hat{\boldsymbol{w}}_{t})\right|\\
& \le\beta_{x}(\left|x_{0}-\bar{x}_{0}\right|,\, t)+\alpha_{w}(\left\Vert \boldsymbol{w}\right\Vert _{0:t-1})+\alpha_{v}(\left\Vert \boldsymbol{v}\right\Vert _{0:t}).
\end{align*}
This means that the FIE is RGAS, and hence completes the RGAS proof.

(b) \emph{Convergence}. Let the sequences of  $\boldsymbol{w}$ and $\boldsymbol{v}$ be bounded as $\|\boldsymbol{w}\|\le M_w$ and $\|\boldsymbol{v}\|\le M_v$ for some constants $M_w, M_v \ge 0$. Since the FIE is RGAS, we have
\begin{align*}
& \left|x(t; x_{0}, \boldsymbol{w}) - x(t; \hat{x}(0|t), \hat{\boldsymbol{w}}_{t})\right| \\
& \le \beta_{x}(\left|x_{0} - \bar{x}_{0}\right|, t) + \alpha_{w}(M_w) + \alpha_{v}(M_v),
\end{align*}
for all $t \ge 0$. Because the disturbance and noise sequences are known to converge to zero, this knowledge constrains the feasible sets of the disturbance and noise estimates and ensures that these estimates obtained by the FIE defined in \eqref{eq: FIE} will converge to zero. Therefore, for any $\epsilon >0$, there exists a time $T_\epsilon > 0$ such that $|w(t)|,|\hat{w}(t)| \le 0.5\alpha_1^{-1}(\epsilon / 8)$ and $|v(t)|,|\hat{v}(t)| \le 0.5\alpha_2^{-1}(\epsilon / 8)$ for all $t \ge T_\epsilon$. By the definition of $\KL$-function, for any $\epsilon > 0$ there exists a time $\tau_\epsilon$ such that $\beta\left(\beta_{x}(\left | x_{0}  -  \bar{x}_0 \right |, 0) + \alpha_{w}(M_w) + \alpha_{v}(M_v), t\right) \le \epsilon/2$ for all $t \ge \tau_\epsilon$. Hence, for $t \ge T_\epsilon + \tau_\epsilon$ we obtain
\begin{align*}
& \left|x(t; x_{0}, \boldsymbol{w}) - x(t; \hat{x}(0|t), \hat{\boldsymbol{w}}_{t})\right| \\
\overset{\text{by i-IOSS}}{\le} &  \beta\left(\left|x(T_\epsilon; x_{0}, \boldsymbol{w}) - x(T_\epsilon; \hat{x}(0|t), \hat{\boldsymbol{w}}_{t})\right|, t-T_\epsilon\right) \\
& + \alpha_{1}(\left\Vert \boldsymbol{w} - \boldsymbol{\hat{w}}_t \right\Vert _{T_\epsilon : t-1}) + \alpha_{2}(\left\Vert \boldsymbol{v} - \boldsymbol{\hat{v}}_t \right\Vert _{T_\epsilon : t})\\
\overset{\text{by RGAS}}{\le} & \beta\left(\beta_{x}(\left|x_0 - \bar{x}_0\right|, 0) + \alpha_w(M_w) + \alpha_v(M_v), t-T_\epsilon\right) \\
& + \alpha_{1}(2\left\Vert \boldsymbol{w}\right\Vert _{T_\epsilon : t-1}) + \alpha_{1}(2\left\Vert \boldsymbol{\hat{w}}_t\right\Vert _{T_\epsilon : t-1})\\
 & + \alpha_{2}(2\left\Vert \boldsymbol{v}\right\Vert _{T_\epsilon : t}) + \alpha_{2}(2\left\Vert \boldsymbol{\hat{v}}_t\right\Vert _{T_\epsilon : t})\\
\le & \epsilon/2 + \epsilon/8  + \epsilon/8  + \epsilon/8 + \epsilon/8 = \epsilon,
\end{align*}
which implies that $x(t; \hat{x}(0|t), \hat{\boldsymbol{w}}_{t})$ converges to $x(t; x_{0}, \boldsymbol{w})$ as $t\rightarrow \infty$. This completes the convergence proof.
\end{proof}
\begin{rem}The RGAS proof is motivated from Proposition 11 of \cite{rawlings2012optimization}, which however can only be applied to the FIE with a specific cost function for convergent disturbances. 

 Theorem \ref{thm:(RGAS-of-FIE} gives sufficient conditions for an FIE to be RGAS (or convergent) under bounded (or convergent) disturbances. From this point of view,  it extends the results in \cite{rawlings2012optimization}. 

\end{rem}

Assumption \ref{assump: A2} is rather general. Using the $\KdL$-function introduced in Definition \ref{def: (-factorizable-},
we can obtain more specific conditions admitting easier interpretation.
\begin{cor}
\label{cor:The-FIE-defined}The FIE defined in (\ref{eq: FIE}) is RGAS if the following conditions are satisfied:

a) the system given in (\ref{eq:system}) is i-IOSS;

b) the FIE's cost function satisfies Assumption \ref{assump: A1}, and the infimum is attainable;

c) the $\KL$-functions $\beta$ in (\ref{eq:definition - i-IOSS}) and $\underbar{\ensuremath{\rho}}_{x},\,\rho_{x}$ in (\ref{eq: Assumption 1}) are $\KdL$-functions in the form of $\beta(s,\, t)=\mu_{1}(s)\varphi_{1}(t)$, $\underbar{\ensuremath{\rho}}_{x}(s,\, t)=\mu_{2}(s)\varphi_{2}(t)$ and $\rho_{x}(s,\,t)=\mu_{3}(s)\varphi_{2}(t)$, where $\mu_{1},\,\mu_{2},\,\mu_{3}\in\K$, and $\varphi_{1},\varphi_{2} \in \LL$, and satisfy
\begin{equation}\label{eq: key-stability-condition}
    \mu_{1}\left(4\mu_{2}^{-1}\left(\frac{\pi(s)}{\varphi_{2}(t)}\right)\right)\varphi_{1}(t) \le \pi'(s)
\end{equation}
for an arbitrary $\pi\in\K$ and some $\pi'\in\K$.
\end{cor}
\begin{proof}
It is sufficient to show that Assumption \ref{assump: A2} is satisfied under the
condition c). With $\underbar{\ensuremath{\gamma}}_{x,\, t}(s):=\underbar{\ensuremath{\rho}}_{x}(s,\, t)=\mu_2(s) \varphi_2(t)$ and $\rho_{x}(s,\,t)=\mu_{3}(s)\varphi_{2}(t)$, we have
\begin{align*}
& \beta\left(4\underbar{\ensuremath{\gamma}}_{x,\, t}^{-1}\left(3\rho_{x}(s_{x},\, t)\right),\, t\right) =  \beta\left(4\mu_{2}^{-1}\left(\frac{3\rho_{x}(s_{x},\, t)}{\varphi_{2}(t)}\right),\, t\right)\\
 & = \beta\left(4\mu_{2}^{-1}\left(3\mu_{3}(s_{x})\right),\, t\right)
 =:\bar{\beta}_{x}(s_{x},\, t),
\end{align*}
which results in a $\KL$-function. With $\beta(s,t)=\mu_1(s)\varphi_1(t)$ and the condition \eqref{eq: key-stability-condition}, we also have,
\begin{align*}
\beta\left(4\underbar{\ensuremath{\gamma}}_{x,\, t}^{-1}\left(3\gamma_{w}(s_{w})\right),\, t\right) & = \mu_{1}\left(4\mu_{2}^{-1}\left(\frac{3\gamma_{w}(s_{w})}{\varphi_{2}(t)}\right)\right)\varphi_{1}(t)\\
  & \le \bar{\alpha}_{w}(s_{w}),
\end{align*}
for some function $\bar{\alpha}_{w}\in\K$. Similarly we have $\beta\left(4\underbar{\ensuremath{\gamma}}_{x,\, t}^{-1}\left(3\gamma_{v}(s_{v})\right),\, t\right)\le \bar{\alpha}_{v}(s_{v})$,
for some function $\bar{\alpha}_{v}\in\K$. Consequently,
\begin{align*}
 & \beta\left(s_{x}+\underbar{\ensuremath{\gamma}}_{x,\, t}^{-1}\left(\rho_{x}(s_{x},\, t)+\gamma_{w}(s_{w})+\gamma_{v}(s_{v})\right),\, t\right)\\
 & \le \beta\left(\begin{array}{c}
s_{x}+\underbar{\ensuremath{\gamma}}_{x,\, t}^{-1}\left(3\rho_{x}(s_{x},\, t)\right)\\
+\underbar{\ensuremath{\gamma}}_{x,\, t}^{-1}\left(3\gamma_{w}(s_{w})\right)+\underbar{\ensuremath{\gamma}}_{x,\, t}^{-1}\left(3\gamma_{v}(s_{v})\right)
\end{array},\, t\right)\\
 & \le \beta\left(4s_{x},\, t\right)+\beta\left(4\underbar{\ensuremath{\gamma}}_{x,\, t}^{-1}\left(3\rho_{x}(s_{x},\, t)\right),\, t\right)\\
 &  \,\,\,\,\,\,+\beta\left(4\underbar{\ensuremath{\gamma}}_{x,\, t}^{-1}\left(3\gamma_{w}(s_{w})\right),\, t\right)+\beta\left(4\underbar{\ensuremath{\gamma}}_{x,\, t}^{-1}\left(3\gamma_{v}(s_{v})\right),\, t\right)\\
 & \le \beta\left(4s_{x},\, t\right)+\bar{\beta}_{x}(s_{x},\, t)+\bar{\alpha}_{w}(s_{w})+\bar{\alpha}_{v}(s_{v})\\
 & = \bar{\beta}_{x}'(s_{x},\, t)+\bar{\alpha}_{w}(s_{w})+\bar{\alpha}_{v}(s_{v}),
\end{align*}
where $\bar{\beta}_{x}'(s_{x},\, t):=\beta\left(4s_{x},\, t\right)+\bar{\beta}_{x}(s_{x},\, t)$
which is a $\KL$-function. The last inequity means that
Assumption \ref{assump: A2} is satisfied. Together with the conditions in a) and b),
it establishes the conclusion by using Theorem \ref{thm:(RGAS-of-FIE}.
\end{proof}

In the condition c) of Corollary \ref{cor:The-FIE-defined}, the assumption of $\beta$ being a $\KdL$-function is trivial because we can always assign such a function as an alternative if the original $\KL$-function $\beta$ is not in a $\KdL$ form (cf. Lemma \ref{lem: factorizable-KL-bound}). The condition that $\underbar{\ensuremath{\rho}}_{x}$ and $\rho_{x}$ in \eqref{eq: Assumption 1} are $\KdL$-functions is not on the system dynamics, but a requirement on the cost function defined for the FIE. The key condition thus boils down to \eqref{eq: key-stability-condition}, requiring the cost function to be sufficiently sensitive (compared to the system's sensitivity) to the uncertainty in the initial state. This is intuitive because otherwise the estimator cannot detect the effect caused by the uncertainty and hence is unable to reconstruct the initial state accurately.

The FIE admits a more specific cost function if the system is i-IOSS as in \eqref{eq:definition - i-IOSS} where the $\KL$ bound has a polynomial form.

\begin{cor}
\label{cor:The-FIE-defined2}
The FIE defined in (\ref{eq: FIE}) is RGAS, if the following conditions are satisfied:

a) the system (\ref{eq:system}) is i-IOSS with the $\KL$ bound being given as $\beta(s,t)=c_{1}s^{a_1}(t+1)^{-b_1}$ for some constants $c_{1},a_{1},b_{1}>0$ and all $s,t\ge0$;

b) the infimum in (\ref{eq: FIE}) is attainable when the cost function is defined as
\begin{equation*}
    V_{t}(\mathcal{X}(0)-\bar{x}_0,\,\boldsymbol{\omega})
    = l_{x}(\mathcal{X}(0)-\bar{x}_0)(t+1)^{-b_{2}}+ l_{wv}(\boldsymbol{\omega},\,\boldsymbol{\nu},\,t),
\end{equation*}
where $b_{2}$ is a positive constant, and the functions $l_{x}$ and $l_{wv}$ are continuous and satisfy the following inequalities for all $x\in \mathbb{R}^{n}$, $\boldsymbol{w}\in \mathbb{B}_{w}$ and $\boldsymbol{v}\in \mathbb{B}_{v}$:
\begin{align*}
c_{2}|x|^{a_2}=:\underbar{\ensuremath{\gamma}}_{x}'(|x|) & \le l_{x}(x)\le \gamma_{x}'(|x|),\\
\underbar{\ensuremath{\gamma}}_{w}(\|\boldsymbol{w}\|_{0:t-1})+\underbar{\ensuremath{\gamma}}_{v}(\|\boldsymbol{v}\|_{0:t}) & \le l_{wv}(\boldsymbol{w},\,\boldsymbol{v},\,t)\\
& \le \gamma_{w}(\|\boldsymbol{w}\|_{0:t-1})+\gamma_{v}(\|\boldsymbol{v}\|_{0:t}),
\end{align*}
in which $c_{2}$ and $a_2$ are positive constants, and $\underbar{\ensuremath{\gamma}}_{w},\,\underbar{\ensuremath{\gamma}}_{v}, \,\gamma_{x}',\,\gamma_{w}, \,\gamma_{v} \in \Kinf$;

c) the above parameters $a_{2}$ and $b_{2}$ satisfy $\frac{a_2}{b_2}\ge\frac{a_1}{b_1}$.
\end{cor}
\begin{proof}
It is straightforward to show that the cost function given above satisfies Assumption \ref{assump: A1}, in which the $\KL$-functions are given as $\underbar{\ensuremath{\rho}}_{x}(|x|,\,t):=\underbar{\ensuremath{\gamma}}_{x}'(|x|)(t+1)^{-b_2}$ and $\rho_{x}:=\gamma_{x}'(|x|)(t+1)^{-b_2}$. These two functions are factorizable as $\mu_{2}(s)\varphi_{2}(t)$ and $\mu_{3}(s)\varphi_{2}(t)$, respectively, with $\mu_{2}(s):=\underbar{\ensuremath{\gamma}}_{x}'(s)=c_{2}s^{a_2}$, $\mu_{3}(s):=\gamma_{x}'(s)$ and $\varphi_{2}(t):=(t+1)^{-b_2}$. Given the condition a) above, the $\KL$ bound associated with the i-IOSS property of the system is obtained as $\beta_{x}(s,\,t)=\mu_{1}(s)\varphi_{1}(t)$, with $\mu_{1}(s):=c_{1}s^{a_1}$ and $\varphi_{1}(t):=(t+1)^{-b_1}$. Then for any function $\pi\in\K$, we have \begin{align*}
& \mu_{1}\left(4\mu_{2}^{-1}\left(\frac{\pi(s)}{\varphi_2(t)}\right)\right)\varphi_{1}(t)\\
& =c_{1}\left(4\left(\frac{1}{c_2}\pi(s)(t+1)^{b_2}\right)^{\frac{1}{a_2}}\right)^{a_1}(t+1)^{-b_1}\\
& =4^{a_{1}}c_{1}c_{2}^{-\frac{a_1}{a_2}}(\pi(s))^{\frac{a_1}{a_2}}(t+1)^{\frac{a_1b_2}{a_2}-b_1}\\
& \le 4^{a_{1}}c_{1}c_{2}^{-\frac{a_1}{a_2}}(\pi(s))^{\frac{a_1}{a_2}}=:\pi'(s),
\end{align*}
where $\pi'$ is a $\KL$-function, and the condition c) has been used to derive the inequality. Hence the condition c) of Corollary \ref{cor:The-FIE-defined} is satisfied. As the conditions a) and b) there are also satisfied, this proves that the FIE is RGAS by Corollary \ref{cor:The-FIE-defined}.
\end{proof}

The conditions b)-c) of Corollary \ref{cor:The-FIE-defined2} are manifestations of the general conditions given in Theorem \ref{thm:(RGAS-of-FIE}, subject to the condition a) here. We remark that this corollary recovers the main result in \cite{ji2013robust} if the design parameter $b_2$ is fixed to 1 (with a minor difference that here the FIE is able to utilize the last measurement in the estimation, whose fitting error is penalized through $\nu(t)$).

More specific cost functions that satisfy the conditions b)-c) of Corollary \ref{cor:The-FIE-defined2} may have the following forms:
\begin{align*}
l_{x}(\mathcal{X}(0)-\bar{x}_0)(t+1)^{-b_2}:=\frac{c_2|\mathcal{X}(0)-\bar{x}_0|^{a_2}}{(t+1)^{b_2}},
\end{align*}
for positive constants $a_{2},b_{2}$ satisfying $\frac{a_2}{b_2}\ge\frac{a_1}{b_1}$ and any positive constant $c_2$, and
\begin{align*}
&l_{wv}(\boldsymbol{\omega},\,\boldsymbol{\nu},\,t)
=\dfrac{1}{t+1}\left(\lambda_{w}\sum_{i=0}^{t-1}l_{w,i}(\omega(i))+\lambda_{v}\sum_{i=0}^{t}l_{v,i}(\nu(i))\right)\\
&\,\,\,\,+(1-\lambda_{w})\max_{i\in \mathbb{I}_{0:t-1}}{l_{w,i}(\omega(i))}+(1-\lambda_{v})\max_{i\in \mathbb{I}_{0:t}}{l_{w,i}(\nu(i))},
\end{align*}
for given constants $\lambda_{w},\,\lambda_{v}\in[0,\,1]$, in which the functions $l_{w,i}$ and $l_{v,i}$ are such that:
\begin{equation*}
\underbar{\ensuremath{\gamma}}_{w}'(|w|) \le l_{w,i}(w)\le \gamma_{w}'(|w|),\,\,
\underbar{\ensuremath{\gamma}}_{v}'(|v|) \le l_{v,i}(v)\le \gamma_{v}'(|v|),
\end{equation*}
where $\underbar{\ensuremath{\gamma}}_{w}',\,\underbar{\ensuremath{\gamma}}_{v}',\,\gamma_{w}',\,\gamma_{v}'\in \Kinf$.

Furthermore, if the system described in \eqref{eq:system} is exp-i-IOSS, then the polynomial $\KL$ bound in Corollary \ref{cor:The-FIE-defined2} can be tightened to have an exponential form. Consequently we may define a cost that better penalizes the deviation from the prior initial state estimate, which intuitively would improve FIE's estimation performance.

\begin{cor} \label{cor:The-FIE-defined3}
The FIE defined in (\ref{eq: FIE}) is RGAS, if the following conditions are satisfied:

a) the system (\ref{eq:system}) is exp-i-IOSS with the $\KL$-function being given as $\beta(s,\,t)=c_1 s^{a_1} b_1^t$ for some constants $c_1, a_1 > 0$ and $0 < b_1 < 1$ and all $s,t \ge 0$;

b) the condition b) of Corollary \ref{cor:The-FIE-defined2} is satisfied with the factor $(t+1)^{-b_2}$ being replaced with $b_2^t$;

c) the  parameters $a_{2}$ and $b_{2}$ satisfy $\sqrt[a_2]b_2 \ge \sqrt[a_1]b_1$.
\end{cor}
\begin{proof}
The proof follows a routine similar to that of the proof for Corollary \ref{cor:The-FIE-defined2} and is omitted for brevity.
\end{proof}

It is worthwhile to mention that the condition c) of Corollary \ref{cor:The-FIE-defined3} does not require $b_2 < 1$. That is, an FIE with $b_2 \ge 1$ may also be RGAS despite that the sub-cost associated with the initial state diverges in time. We will illustrate this in the simulation section.

By Corollary \ref{cor:The-FIE-defined3}, it is valid to specify the sub-cost associated with the initial state as $c_2|\mathcal{X}(0)-\bar{x}_0|^{a_2} b_2^t$, with the positive constants $\{a_2, b_2, c_2\}$ satisfying the condition c) of Corollary \ref{cor:The-FIE-defined3}. The sub-cost associated with the disturbances may be defined to have the same form presented after Corollary \ref{cor:The-FIE-defined2}.

\begin{rem}
As in \cite{glas1987exponential}, nonlinear systems that are asymptotically stable but not exponentially stable fail to be structurally stable and constitute a boundary set, and hence of little practical interest. They prove that the set of exponentially stable systems are dense in the whole set of asymptotically stable systems. It thus does not lose generality or practical interest for Corollary \ref{cor:The-FIE-defined3} to focus on i-IOSS systems that are exponentially stable.
\end{rem}

\begin{rem}
The conclusion that the state estimate given by the FIE converges to the true state if we know the disturbances converge to zero remains true under the conditions of Corollaries \ref{cor:The-FIE-defined}-\ref{cor:The-FIE-defined3}. This is because the convergence is implied by the i-IOSS property of the system and the RGAS property of the estimator under bounded and convergent disturbances (cf. the proof of Theorem \ref{thm:(RGAS-of-FIE}).
\end{rem}

\section{Numerical Example}\label{sec:numerical example}

We use a simple example to illustrate the theoretical
results concluded by Corollaries \ref{cor:The-FIE-defined2}-\ref{cor:The-FIE-defined3}. Consider an asymptotically stable system with linear dynamics and nonlinear measurement: $x^{+}=0.9x+w,\,y=x^3+v$, where $x$ is the state, $y$ the measurement, $w$ the state disturbance, and $v$ the measurement noise. The disturbance $\{w(k)\}$ and noise $\{v(k)\}$ are two sequences of independent, zero
mean, normally distributed random variables with variances $\sigma_{w}^{2}$ and $\sigma_{v}^{2}$ equal to $0.1^2$ and $0.2^2$, respectively, as
further truncated to the intervals $[-3\sigma_{w},\,3\sigma_{w}]$ and $[-3\sigma_{v},\,3\sigma_{v}]$, respectively. The initial state $x(0)$ is a random variable independent of the disturbances $\{w(k)\}$ and $\{v(k)\}$, and follows a normal distribution with a mean of 5 and a variance of $\sigma_{x_{0}}^{2}$ equal to 4. The prior estimate of the initial state is given as $\bar{x}_{0}=2$.

We can show that the system is exp-i-IOSS with the $\K\cdot\LL$ bound given by $\beta(s, t) = s0.9^t$ (the proof is omitted for page limit). By Corollary \ref{cor:The-FIE-defined3}, for the FIE to be RGAS its cost function can be specified as
\begin{align*}
V_{t}=& \frac{\left(\chi(0)-\bar{x}_{0}\right)^{2}b_2^t}{\sigma_{x_{0}}^{2}}
+ \dfrac{1}{t+1}\left(\dfrac{\lambda_{w}}{\sigma_{w}^{2}}\sum_{i=0}^{t-1}\omega^{2}(i)+\dfrac{\lambda_{v}}{\sigma_{v}^{2}}\sum_{i=0}^{t}\nu^{2}(i)\right)\\
&+\dfrac{1-\lambda_{w}}{\sigma_{w}^{2}}\|\boldsymbol{\omega}\|_{0:t-1}^{2}+\dfrac{1-\lambda_{v}}{\sigma_{v}^{2}}\|\boldsymbol{\nu}\|_{0:t}^{2},
\end{align*}
for any given constants $b_2 \ge 0.9^2=0.81 $ and $\lambda_{w},\,\lambda_{v}\in[0,\,1]$. By solving the FIE (with $b_2=0.81$) subject to $\left\Vert \boldsymbol{\omega}\right\Vert _{0:t-1} \le 3\sigma_{w}$ and
$\left\Vert \boldsymbol{\nu}\right\Vert _{0:t} \le 3\sigma_{v}$, we obtain the state estimates for each $t\in \mathbb{I}_{0:20}$. The estimation errors, defined by $e(t|t)=x(t)-\hat{x}(t|t)$, are averaged over 500 random instances, as shown in Fig. \ref{fig: numerical results} for evenly sampled times. To compare, the state estimation errors resulting from a generic EKF \cite{haseltine2005critical} are also shown in the figure.

\begin{figure}
\begin{centering}
\includegraphics[scale=0.65]{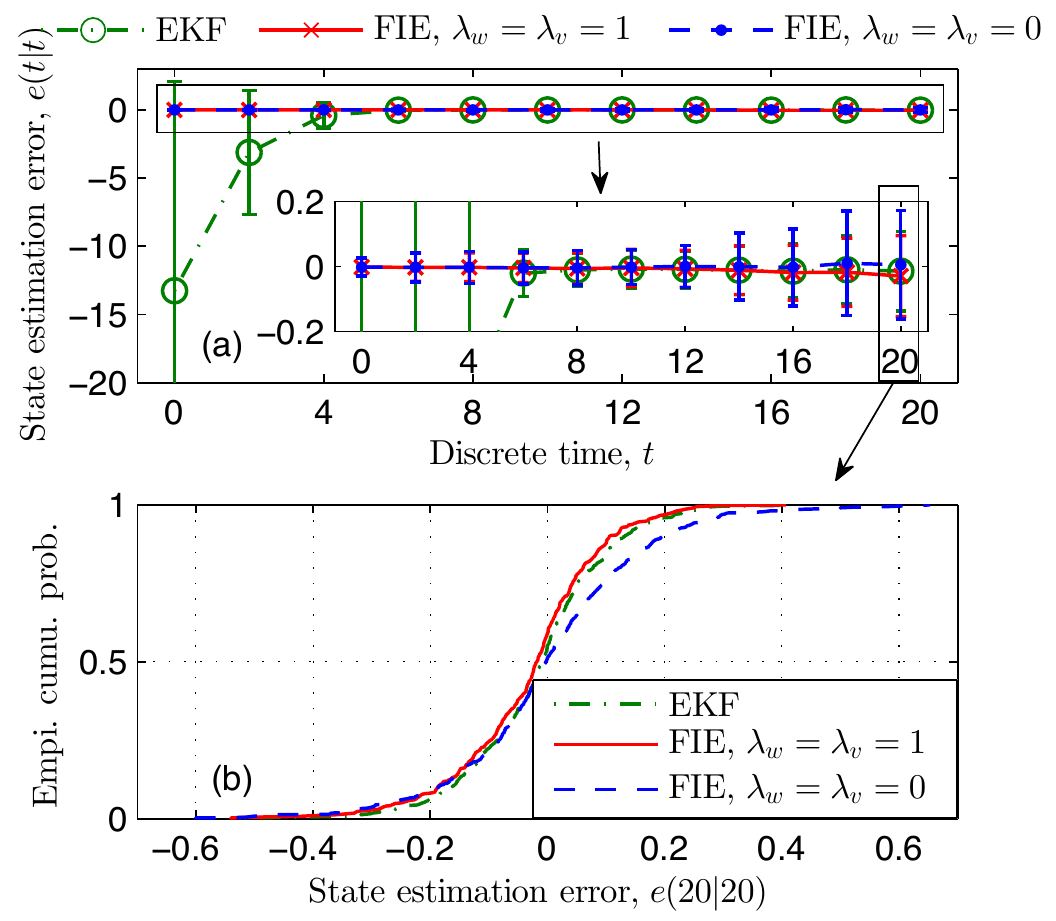}
\par\end{centering}
\caption{ Estimation results: (a) mean and variation of the state estimation error;
(b) empirical cumulative distribution function plot of the state estimation error.}
\label{fig: numerical results}
\end{figure}

We observe that the FIE yields bounded estimation errors (which holds true for longer simulation times) and outperforms the EKF significantly during the early estimation stage. Yet the advantage decays as the EKF accumulates sufficient iterations, say, when $t\ge6$ in this case. The early advantage owes to FIE using all measurements accumulated to compute an optimal estimate of the present state, while the EKF merely uses the current measurement to update its previous estimate. From Fig. \ref{fig: numerical results}, we also observe that the FIE with $\lambda_w=\lambda_v=1$ results in more accurate estimation than with $\lambda_w=\lambda_v=0$. Moreover, we applied the FIEs with $b_2=2$ for the two cases, which imposes a heavier and divergent sub-cost for deviation of the estimate of the initial state from its prior estimate. The FIEs yield slightly worse estimation results: when $b_2=0.81$, the state estimation error has a standard deviation of 0.065 (or 0.081) and an average absolute value of 0.037 (or 0.046) over $t\in \mathbb{I}_{0:20}$ for $\lambda_w=\lambda_v=1$ (or 0); in contrast, when $b_2=2$ the standard deviation and the average absolute value are equal to 0.068 and 0.038 (or, 0.091 and 0.057 for $\lambda_w = \lambda_v = 0$), respectively.

 Additionally, we may use a looser $\K\cdot\LL$ bound as $\beta(s, t) = s(t+1)^{\ln0.9}$, and consequently the cost function can alternatively be defined by Corollary \ref{cor:The-FIE-defined2} as:
\begin{align*}
V_{t}'=&\frac{\left(\chi(0)-\bar{x}_{0}\right)^{2}}{\sigma_{x_{0}}^{2}(t+1)^{b_2}}
+ \dfrac{1}{t+1}\left(\dfrac{\lambda_{w}}{\sigma_{w}^{2}}\sum_{i=0}^{t-1}\omega^{2}(i)+\dfrac{\lambda_{v}}{\sigma_{v}^{2}}\sum_{i=0}^{t}\nu^{2}(i)\right)\\
&+\dfrac{1-\lambda_{w}}{\sigma_{w}^{2}}\|\boldsymbol{\omega}\|_{0:t-1}^{2}+\dfrac{1-\lambda_{v}}{\sigma_{v}^{2}}\|\boldsymbol{\nu}\|_{0:t}^{2},
\end{align*}
where $0<b_2 \le -2\ln0.9 \approx 0.21$ and $\lambda_{w},\,\lambda_{v}$ are the same as above. We implemented the FIE with this new cost function for $b_2=0.21$ and ran simulations on the same instances. The state estimation results almost coincide with those obtained using the previous cost function for $b_2=0.81$: the standard deviation and the average absolute value are obtained as 0.065 and  0.037 (or, 0.082 and 0.046 for $\lambda_w = \lambda_v = 0$), respectively.

\section{Conclusions\label{sec:Conclusion}}

This paper presented sufficient conditions for a
full information estimator (FIE) to be robustly globally asymptotically
stable (RGAS) under bounded process and measurement disturbances. The conditions require that the cost function
being optimized has a property resembling the i-IOSS
stability of the system, but with a higher sensitivity to the uncertainty of the initial state. The results are applicable to convergent disturbances, yielding a stronger conclusion that the estimation error of the FIE converges to zero.

As the FIE becomes computationally intractable once the estimation time is large, it is practically important to extend our results to the moving-horizon estimator (MHE). Intuitively, this would require stringer conditions than those of the FIE. The future research is thus to establish such conditions and prove their sufficiency and/or necessity.



\renewcommand{\bibfont}{\fontsize{9}{10}\selectfont}

\bibliographystyle{IEEEtran}
\bibliography{HuXieYou_FIE2015}

\end{document}